\documentclass[12pt]{amsart}

\usepackage{amsfonts}
\usepackage{amssymb}

\newtheorem{theorem}{Theorem}

\newcommand{\Tr}{{\rm tr}\,}

\newcommand{\bR}{\mathbb{R}}

\newcommand{\bN}{\mathbb{N}}
\newcommand{\bF}{\mathbb{F}}

\begin{document}

\baselineskip 7mm

\title{When powers of a matrix coincide with its Hadamard powers} 

\author{Roman Drnov\v sek}


\begin{abstract}
We characterize  matrices whose powers coincide with their Hadamard powers.
\end {abstract}

\maketitle

{\it Math. Subj. Classification (2010)}: 15A21 \\

Let $M_n(\bF)$ be the algebra of all $n \times n$ matrices over the field $\bF$. 
The {\it Hadamard product} of matrices $A = [a_{i j}]_{i, j = 1}^n \in M_n(\bF)$ and $B = [b_{i j}]_{i, j = 1}^n \in M_n(\bF)$
is the matrix $A \circ B = [a_{i j} b_{i j}]_{i, j = 1}^n$. The usual product of $A$ and $B$ is denoted by $A B$.
Given a positive integer $r$, the {\it $r$-th Hadamard power} of a matrix 
$A = [a_{i j}]_{i, j = 1}^n \in M_n(\bF)$ is the matrix $A^{(r)} = [a_{i j}^r]_{i, j = 1}^n$, while the usual $r$-th power of $A$ is denoted by $A^{r}$. 

Let  $p (\lambda) = c_m \lambda^m+ c_{m-1} \lambda^{m-1} + \cdots + c_1 \lambda$ be a polynomial 
with given coefficients $c_m, c_{m-1}, \ldots, c_1 \in \bF$ and without constant term. For any $A \in M_n(\bF)$, 
we can first define the usual matrix function by 
$$ p (A) = c_m A^m+ c_{m-1} A^{m-1} + \cdots + c_1 A , $$
and then also the Hadamard matrix function by
$$ p^H (A) = c_m A^{(m)}+ c_{m-1} A^{(m-1)} + \cdots + c_1 A . $$
The Hadamard product and Hadamard matrix functions arise naturally in a variety of ways (see e.g. \cite[Section 6.3]{HJ}).  
So, it is perhaps useful to know for which matrices $A$ we have $p (A) =  p^H (A)$ for all such polynomials $p$, or equivalently, 
$A^r = A^{(r)}$ for every $r \in \bN$.  The latter question has been recently posed in \cite{Om} for the case of real matrices, 
and two characterizations of such matrices have been given in \cite{He} and \cite{Ku}. 
In this note we give another description of such matrices. 

\begin{theorem}
\label{projections} 
Let $A \in M_n(\bF)$ be a nonzero matrix. Then the following assertions are equivalent:

(a) $A^r = A^{(r)}$ for every positive integer $r$; 

(b) $A^r = A^{(r)}$ for every integer $r \in \{2, 3, \ldots, n+1\}$;

(c) There exist $k \in \bN$, distinct non-zero elements $\lambda_1$, $\ldots$, $\lambda_k \in \bF$, 
and idempotent $(0,1)$-matrices $E_1$, $\ldots$, $E_k$ such that 
$$ A = \sum_{i=1}^k \lambda_i E_i \ \ \ \textrm{and} $$
$$ E_i \circ E_j = E_i E_j = 0 \ \ \  \textrm{for all} \ \ i \neq j . $$
\end{theorem}

\begin{proof}
The implication $(a) \Rightarrow (b)$ is trivial. We begin the proof of the implication $(b) \Rightarrow (c)$
by letting $p (\lambda) = c_m \lambda^m+ c_{m-1} \lambda^{m-1} + \cdots + c_1 \lambda$ be a polynomial of degree $m \le n+1$. If 
$A = (a_{i j})_{i, j = 1}^n$, then our assumptions give that 
\begin{equation}
p(A) = c_m A^m+ c_{m-1} A^{m-1} + \cdots + c_1 A = c_m A^{(m)}+ c_{m-1} A^{(m-1)} + \cdots + c_1 A = [p(a_{i j})]_{i, j = 1}^n \ . 
\label{equal}
\end{equation}
This implies that 
\begin{equation}
p(A) = 0  \iff  p(a_{i j}) = 0 \ \ \ \text{for all} \ \   i, j \ .
\label{zero}
\end{equation}
Let $m(\lambda)$ be the minimal polynomial of $A$. If $A$ is invertible, put $p(\lambda) = \lambda \, m(\lambda)$,
otherwise let $p(\lambda) = m(\lambda)$. Then the degree of $p(\lambda)$ is at most $n+1$ and $p(0) = 0$, so that 
the equivalence (\ref{zero}) implies that  $p(a_{i j}) = 0$ for all $i, j$.
Let $q(\lambda)$ be the minimal polynomial annihilating the element $0$ and all entries of $A$.
Then the polynomial $q(\lambda)$ divides the polynomial $p(\lambda)$, so that its degree is at most $n+1$.
Therefore, the equivalence (\ref{zero}) gives that $q(A) = 0$, and so $m(\lambda)$ divides $q(\lambda)$,
as $m(\lambda)$ is the minimal polynomial of $A$. 
This means that $m(\lambda)$ factors into distinct linear factors, 
so that the matrix $A$ is diagonalizable over $\bF$ and the set $\{\lambda_1, \ldots, \lambda_k \}$ of all non-zero eigenvalues of $A$ coincides with the set of all non-zero entries of $A$.

Now, for each $i =1, \ldots, k$, let $p_i(\lambda)$ be the Lagrange interpolation polynomial such that $p_i(\lambda_i) = 1$,
$p_i(0) = 0$, and $p_i(\lambda_j)= 0$ for all $j \neq i$, that is,
$$ p_i(\lambda) = \frac{\lambda (\lambda - \lambda_1) \cdots (\lambda - \lambda_{i-1}) (\lambda - \lambda_{i+1}) \cdots
(\lambda - \lambda_k)}{\lambda_i (\lambda_i - \lambda_1) \cdots (\lambda_i - \lambda_{i-1}) (\lambda_i - \lambda_{i+1}) \cdots
(\lambda_i - \lambda_k)}  \ . $$
Then $E_i = p_i(A)$ is an idempotent. Furthermore, 
$E_i E_j = 0$ for all $i \neq j$, as $p_i(\lambda) p_j(\lambda) = 0$ on the spectrum of the diagonalizable matrix $A$.
Since each entry of $A$ belongs to the set $\{0, \lambda_1, \ldots, \lambda_k \}$, it follows from (\ref{equal}) that 
$E_1$, $\ldots$, $E_k$ are $(0,1)$-matrices satisfying $E_i \circ E_j = 0$ for all $i \neq j$.
Finally, since $\lambda = \sum_{i=1}^k \lambda_i p_i(\lambda)$ on the spectrum of the diagonalizable matrix $A$, we conclude that 
$A = \sum_{i=1}^k \lambda_i p_i(A) = \sum_{i=1}^k \lambda_i E_i$. This proves the implication $(b) \Rightarrow (c)$.

For the proof of the implication $(c) \Rightarrow (a)$, we just compute the powers:
$$ A^r = \sum_{i=1}^k \lambda_i^r E_i = A^{(r)} . $$
\end{proof}

We now give the canonical form of an idempotent $(0,1)$-matrix. 
When the field $\bF$ is the field $\bR$ of all real numbers, this can be obtained easily from the canonical form 
of a nonnegative idempotent matrix (see e.g. \cite[Theorem 3.1 on page 65]{BP}). 

\begin{theorem}
\label{form} 
Let $E \in M_n(\bF)$ be an idempotent $(0,1)$-matrix of rank $m \in \bN$. 
Suppose that the characteristic of the field $\bF$ is either zero or larger than $n$. 
Then there exists a permutation matrix $P$ such that 
$$ P E P^T = \left[ \begin{matrix}
I & U  & 0 & 0 \cr
0 & 0  & 0 & 0 \cr
V & VU & 0 & 0 \cr
0 & 0  & 0 & 0 
\end{matrix}  \right] = 
\left[ \begin{matrix}
I \cr
0 \cr
V \cr
0  
\end{matrix}  \right] \cdot 
\left[ \begin{matrix}
I & U  & 0 & 0 
\end{matrix}  \right] , $$
where $I$ is the identity matrix of size $m$, and $U$, $V$ are $(0,1)$-matrices such that
$U$ has no zero columns, $V$ has no zero rows, and $V U$ is also  a $(0,1)$-matrix.
(It is possible that $U$ or $V$ act on zero-dimensional spaces.)
\end{theorem}

\begin{proof}
Suppose first that $E$ has no zero rows and no zero columns. We must show that $m = n$ and $E=I$.
Assume on the contrary that $m < n$. Since $\Tr (E) = m$ and $E$ is a $(0,1)$-matrix, there exists a permutation matrix $P$ such that 
$$ P E P^{\rm T} = \left[ \begin{matrix}
A & B \cr
C & D
\end{matrix}  \right] , $$
where the diagonal entries of $A \in M_m(\bF)$ are equal to $1$, while the diagonal entries of $D \in M_{n-m}(\bF)$ are equal to $0$.
Since $E$ is an idempotent, we have $A^2 + B C = A$, so that, in view of our characteristic assumption, $A^2$ is also a $(0,1)$-matrix. 
It follows that $A$ must be the identity matrix. As $P E P^{\rm T}$ is an idempotent, we obtain that 
$BC=0$, $BD=0$, $DC=0$ and $CB + D^2 = D$. 
Since $E$ has no zero rows, the equalities $BC=0$ and $BD=0$ imply that $B=0$. 
Since $D^2 = D$ and $\Tr (D) = 0$, we conclude that $D=0$. This contradicts the fact that $E$ has no zero columns. So, we must have that $m = n$ and $E=I$.

To prove the general case, let us group the indices $i=1, 2, \ldots, n$ into four sets according to whether 
the $i$-th row and the $i$-th column of $E$ are both non-zero, or the $i$-th row is zero but the $i$-th column is not, and so on. 
So, there exists a permutation matrix $P$ such that 
$$ P E P^{\rm T} = \left[ \begin{matrix}
T & U & 0 & 0 \cr
0 & 0 & 0 & 0 \cr
V & W & 0 & 0 \cr
0 & 0 & 0 & 0 
\end{matrix}  \right] , $$
where $T$, $U$, $V$, $W$ are $(0,1)$-matrices such that $T$ and $U$ have no zero rows in common, 
and $T$ and $V$ have no zero columns in common.
Since $E^2 = E$, we have $T^2 = T$, $T U = U$, $V T = V$ and $V U = W$. 
It follows from  $W = V U$ that $U$ has no zero columns and $V$ has no zero rows.
Indeed, if $U$ had a zero column, then the whole column in  $P E P^{\rm T}$ would be zero, cotradicting the definition of the second group of indeces.
As $T$ and $U = T U$ have no zero rows in common, $T$ has no zero rows. 
Similarly, $T$ cannot have a zero column. By the first part of the proof, we obtain that $T=I$ which gives the desired form.
\end{proof}

In Theorem \ref{form} we cannot omit the assumption on the characteristic of the field $\bF$. Namely, if the field $\bF$ has prime characteristic 
$p < n$, then, for example, take the $(p+1) \times (p+1)$ matrix of all ones and enlarge it by adding zeros to get 
an idempotent $(0,1)$-matrix in $M_n(\bF)$ which is not of the above form.

If we apply Theorem \ref{form} for idempotent $(0,1)$-matrices in the assertion (c) of Theorem \ref{projections}, we obtain the following 
decription of a matrix whose powers coincide with its Hadamard powers.

\begin{theorem}
\label{real} 
Let $A \in M_n(\bF)$ be a non-zero matrix of rank $m$, where the characteristic of the field $\bF$ is either zero or larger than $n$.
Then the assertions (a), (b) and (c) of Theorem \ref{projections} are further equivalent to the following:

(d) There exist a permutation matrix $P$, non-zero elements $\mu_1$, $\ldots$, $\mu_m \in \bF$, and $(0,1)$-matrices  $U$, $V$ 
such that $U$ has no zero columns, $V$ has no zero rows,  $V U$ is also  a $(0,1)$-matrix, and
$$  P A P^{\rm T} = \sum_{i=1}^m \mu_i v_i u_i^{\rm T} , $$
where $u_1^{\rm T}$,  $\ldots$, $u_m^{\rm T}$ are the rows of the $m \times n$ matrix 
$\left[ \begin{matrix}
I & U  & 0 & 0 
\end{matrix}  \right]$, and $v_1$,  $\ldots$, $v_m$ are the columns of the  $n \times m$ matrix 
$ \left[ \begin{matrix}
I  & 0  & V  &  0  
\end{matrix}  \right]^{\rm T} $. 

\end{theorem}

\begin{proof}
We must explain only how to obtain the assertion (d) from the assertion (c) of Theorem \ref{projections}. 
We first observe that the matrix $E = E_1 + \cdots + E_k$ is an idempotent $(0,1)$-matrix of rank $m$. 
By Theorem \ref{form}, there is a permutation matrix $P$ such that 
$$ P E P^{\rm T} = \left[ \begin{matrix}
I & U  & 0 & 0 \cr
0 & 0  & 0 & 0 \cr
V & VU & 0 & 0 \cr
0 & 0  & 0 & 0 
\end{matrix}  \right] = 
\left[ \begin{matrix}
I \cr
0 \cr
V \cr
0  
\end{matrix}  \right] \cdot 
\left[ \begin{matrix}
I & U  & 0 & 0 
\end{matrix}  \right] , $$
where $I$ is the identity matrix of size $m$, and $U$, $V$ are $(0,1)$-matrices such that
$U$ has no zero columns, $V$ has no zero rows, and $V U$ is also  a $(0,1)$-matrix.
Let $u_1^{\rm T}$,  $\ldots$, $u_m^{\rm T}$ be the rows of the matrix 
$\left[ \begin{matrix}
I & U  & 0 & 0 
\end{matrix}  \right]$, and let $v_1$,  $\ldots$, $v_m$ be the columns of the matrix 
$ \left[ \begin{matrix}
I  & 0  & V  &  0  
\end{matrix}  \right]^{\rm T} $.
Then $u_i^{\rm T}  v_i = 1$ for all $i$, $u_i^{\rm T}  v_j = 0$ for all $i \neq j$, and 
$P E P^{\rm T} = \sum_{i=1}^m v_i u_i^{\rm T}$.  
Since  $E_1$, $ \ldots$, $E_k$ and $E = E_1 + \cdots + E_k$ are $(0,1)$-matrices, 
all the ones of a matrix $E_j$ ($j = 1, \ldots, k$) are at positions where also $E$ has ones. Thus, 
we have 
$$ P E_j P^{\rm T} = \left[ \begin{matrix}
I_j & U_j  & 0 & 0 \cr
0 & 0  & 0 & 0 \cr
V_j & (V U)_j & 0 & 0 \cr
0 & 0  & 0 & 0 
\end{matrix}  \right] \  , $$
where $U_j$ is a matrix obtained from $U$ by replacing some ones with zeros, and likewise for $I_j$, $V_j$, and $(V U)_j$.
Now, it follows from $E_j = E_j E = E E_j$ that $U_j = I_j U$ and $V_j = V I_j$, so that each of the first
$m$ rows (resp. columns) of $P E_j P^{\rm T}$ is either equal to 0 or to a corresponding row (resp. column) of $P E P^{\rm T}$. 
Thus,  the matrix $P E_j P^{\rm T}$ is a sum of some of rank-one matrices $v_i u_i^{\rm T}$; for these indices $i$, put $\mu_i = \lambda_j$.  Then we have 
$$  P A P^{\rm T} = \sum_{j=1}^k \lambda_j P E_j P^{\rm T} = \sum_{i=1}^m \mu_i v_i u_i^{\rm T} . $$
\end{proof}

It is worth mentioning that we can eliminate the permutation matrix $P$  in the assertion (d) of Theorem \ref{real} by setting 
$$ \tilde{u}_i = P^{\rm T} u_i \ \ , \ \ \ \tilde{v}_i = P^{\rm T} v_i  \ \ \ \text{for} \ \  i=1, 2, \ldots, m ,  $$
that gives the form
$$  A  = \sum_{i=1}^m \mu_i \tilde{v}_i \tilde{u}_i^{\rm T} . $$
This way, of course, we lose information about the first $m$ coordinates of these vectors. 
As one of the referees noticed, the same form of $A$ can be derived from the theorem in \cite{He},  since an "anchored" submatrix with index set $I \times J$ (introduced there) can be represented by $\mu E$, 
where $E$ is an idempotent $(0,1)$-matrix of the form $v u^{\rm T}$  and $u$, $v$ are  $(0,1)$-vectors associated with the indices in $I$, $J$, respectively.

Finally, we give a simple example showing that the assertion (c) of Theorem \ref{projections} does not imply that, up to a permutation similarity,  $A$ has a block diagonal form with $k$ blocks. 
Given any non-zero real numbers $\alpha$ and $\beta$, define the matrix $A \in M_4(\bR)$ by
$$ A = \left[ \begin{matrix}
\alpha & 0 & 0 & 0 \cr
    0    & \beta  & 0 & \beta \cr
\alpha & \beta & 0 & \beta \cr
0 & 0  & 0 & 0 
\end{matrix}  \right] =  \alpha 
\left[ \begin{matrix}
1\cr
0 \cr
1 \cr
0  
\end{matrix}  \right] \cdot 
\left[ \begin{matrix}
1 & 0  & 0 & 0 
\end{matrix}  \right] + 
\beta 
\left[ \begin{matrix}
0 \cr
1 \cr
1 \cr
0  
\end{matrix}  \right] \cdot 
\left[ \begin{matrix}
0 & 1  & 0 & 1 
\end{matrix}  \right] . $$

\vspace{4mm}
{\bf
\begin{center}
 Acknowledgments
\end{center}
} 
The author was supported in part by the Slovenian Research Agency. He also thanks Bojan Kuzma and the referees for suggestions that improved this paper.

\vspace{2mm}

\baselineskip 6mm
\noindent
Roman Drnov\v sek \\
Department of Mathematics \\
Faculty of Mathematics and Physics \\
University of Ljubljana \\
Jadranska 19 \\
SI-1000 Ljubljana, Slovenia \\
e-mail : roman.drnovsek@fmf.uni-lj.si 


\begin{thebibliography}{9999}

\bibitem{BP} A. Berman, R. J. Plemmons, Nonnegative matrices in the mathematical sciences, 
revised reprint of the 1979 original, Classics in Applied Mathematics 9, 
Society for Industrial and Applied Mathematics (SIAM), Philadelphia, PA, 1994. 

\bibitem{He} E. A. Herman, Solution 50-4.1:  Matrix Power Coefficients, IMAGE, The Bulletin of ILAS, issue 51, pages 38-40, Fall 2013.

\bibitem{HJ} R. Horn, C. R.  Johnson,  Topics in matrix analysis, Cambridge University Press, Cambridge, 1994.

\bibitem{Ku} B. Kuzma, Solution 50-4.2:  Matrix Power Coefficients, IMAGE, The Bulletin of ILAS, issue 51, pages 40-41, Fall 2013.

\bibitem{Om} M. Omarjee, Problem 50-4: Matrix Power Coefficients, IMAGE, The Bulletin of ILAS, issue 50, page 44, Spring 2013.

\end{thebibliography}
\end{document}